\pdfoutput=1
\documentclass[a4paper]{amsart}

\usepackage{layout}
\usepackage[all]{xy}
\usepackage{amsfonts,amsmath,amscd,amsthm,amssymb}

\newcommand{\G}{\mathcal{P}} 
\newcommand{\I}{\mathbb{I}} 
\newcommand{\II}{\mathcal{I}} %
\newcommand{\F}{\mathcal{F}} 
\newcommand{\FF}{\boldsymbol{F}} 
\newcommand{\g}{q} 
\newcommand{\REF}[1]{{\normalfont (\ref{#1})}} 
\newcommand{\virg}[1]{``#1''} 
\newcommand{\R}{\mathbb{R}} 
\newcommand{\df}{\stackrel{\mathrm{def}}{=}} 
\newcommand{\id}{\operatorname{id}} 
\newcommand{\dd}{\operatorname{d}} 
\newcommand{\im}{\operatorname{im}} 
\newcommand{\Hom}{\operatorname{Hom}} 
\newcommand{\Ann}{\operatorname{Ann}} 
\newcommand{\Diff}{\operatorname{Diff}} 
\newcommand{\III}{\boldsymbol{I}} 

\newtheorem{definition}{Definition}
 \newtheorem{lemma}{Lemma}
\newtheorem{remark}{Remark}
\newtheorem{corollary}{Corollary}
\newtheorem{proposition}{Proposition}
\newtheorem{exercise}{Exercise}
\newtheorem{problem}{Problem}
\newtheorem{example}{Example}

\newtheorem{theorem}{Theorem}

\begin{document}

\markboth{G.~Moreno}
{On Families in Differential Geometry}

%
%

\title{ON FAMILIES IN DIFFERENTIAL GEOMETRY}

\author{GIOVANNI MORENO}

\address{Mathematical Institute in Opava\\
Silesian University in Opava\\
Na Rybnicku 626/1, 746 01 Opava, Czech Republic\\
 Giovanni.Moreno@math.slu.cz}

\maketitle


\begin{abstract}
Families of  objects appear in several contexts, like algebraic topology, theory of deformations,     theoretical physics, etc. An unified coordinate--free algebraic framework for {\it families of geometrical quantities} is presented here, which   allows one to work without introducing {\it ad hoc} spaces,  by using  the   language of differential calculus over commutative algebras. An  advantage of such an approach, based on the notion of {\it sliceable structures} on cylinders, is that the fundamental theorems of standard calculus  are straightforwardly generalized to the context of families. As an example of that, we prove the  {\it universal}  homotopy formula.
\end{abstract}


 \subjclass{53C15,58A05,58A10,58H15,58C99}


\maketitle
\maketitle

\tableofcontents

\begin{table}[htdp]
\caption{List of main symbols.}
\begin{center}
\begin{tabular}{r|l}

$M$ & a smooth manifold\\

$\G$ & the manifold of paramaters\\

$\I$ & the closed interval $[0,1]$\\

$\iota_g$ & the slicing map\\

$M_g$ & the slice determined by $g\in\G$\\ 

$\Theta$ & a geometrical quantity on $M\times\G$\\

$\Theta_g$ & restriction of $\Theta$ to the slice $M_g$\\

$\{\Theta_g\}_{g\in \G }$ & family determined by $\Theta$\\

$F^\circ(\pi)$ & pull--back bundle\\

$\pi_M$, $\pi_\G$ & canonical projections\\

$D$ & functor of derivations\\

$D_{\pi}^v$ & functor of $\pi$--vertical derivations\\

$D_{\pi}$ & functor of derivations along $\pi$\\

$\Diff$ & functor of differential operators\\

$\Gamma(\pi)$ & module of sections of $\pi$\\

$\Gamma_c(\pi)$ & submodule of compact--supported sections of $\pi$\\

$\widetilde{X}$ & canonical lift of $X$\\

$\nabla_X$ & der--operator associated with $X$\\

$X(F)$ & the derivative of a family of maps $F$\\

$F'$ & infinitesimal homotopy\\

$\II_\R$, $\II_a^b$ & integration operators\\

$\F$ & a functional on $C^\infty(\G)$\\

$Q^\vee$ & the dual of module $Q$\\

$\pi^\vee$ & the dual of bundle $\pi$\\

$\overline{A}$ & smooth envelope of $A$\\

$\mathrm{Spec}(A)$ & spectrum of $A$\\

$P\overline{\otimes}Q$ & smoothened tensor product\\

$\widetilde{\F}$ & canonical lift of functional $\F$\\

$\FF$ & a functor of diffrential calculus\\

$\Phi$ & representing object of $\FF$\\

$\Phi_F$ & $F$--horizontal  sub--module \\

$\Lambda^1_f$ & $f$--horizontal 1--forms\\

$\Lambda^{p,q}$ & forms of type $(p,q)$\\

$p_{p,q}$ & canonical projector\\

$\II_f$ & ideal of $f$--horizontal forms\\

$\overline{\Lambda}_f$ & algebra of $f$--vertical forms\\

$\overline{d}$ & horizontal differential\\

 \end{tabular}
\end{center}
\label{default}
\end{table}%

\newpage

\section{Introduction}
%


Instances of families of geometrical objects can be found in Differential Geometry, where they play an essential role in the proof of key theorems. They are  relevant in Algebraic Geometry as well, but this is not  touched  upon here. The purpose of this paper is to provide a conceptual approach to the theory of  {\it families of geometrical quantities}---a common denomination which unifies 
   such definitions   as families, deformations, homotopies, isotopies, motions, etc.\footnote{It is left to   the reader the task  to particularize the theory to the cases of personal interest.} We ventured calling it \virg{conceptual} mainly for two reasons. First,  it makes self--evident such a  property as smoothness, which, despite its elementar character, it is usually defined in a slightly cumbersome way. Second, it allows  a rigorous and straightforward generalization of    important  theorems of  differential calculus to the context of   families (see, on this concern, the   universal homotopy formula, proved in Section \ref{secHomFor}).\par
Such an approach cannot be obtained without exploiting  the \emph{logic of differential calculus over commutative algebras}, a theory pioneered by A. M. Vinogradov in the seventies (the main ideas can be found in the   papers \cite{Logic1,Logic2}, while   the book \cite{Nestruev} provides an elementary introduction to the subject). Roughly speaking, this \virg{logic} is composed of the so--called \emph{functors of differential calculus} (e.g., differential operators, derivations, etc.), each of which is accompanied by its representative object (e.g., the module of jets, differential forms, etc.). It turns out that  representative objects are themselves functors, but controvariant ones (unlike the functors of differential calculus, which are covariant ones). However, the main difference between functors and representative objects is that the former are absolute, while the latter are relative, i.e., they depend on the module category in which they are defined. For instance, in order to recover the familiar definition of differential forms over a smooth manifold, one has to introduce the category of geometric modules over smooth algebras. But not only differential forms, even  the whole calculus over smooth manifolds constitutes  a chapter of the logic of differential calculus over commutative algebras, they key being provided by the so--called   {\it Spectral Theorem} (see  \cite{Nestruev}),  an isomorphisms  between  the category of smooth manifolds and the category of smooth algebras. In other words, the logic of differential calculus over commutative algebras allows to formalize   any well--known notion of differential calculus over smooth manifolds   in terms of  objects, morphisms, endofunctors and   their representative objects in the categories of smooth algebras and  geometric modules. But, most importantly, it allows to {\it define} notions (from smoothness itself) and theorems (e.g., the Newton--Leibniz formula) of differential calculus in far more general contexts than smooth manifolds (see, for instance, \cite{ActaSpeciale}) and, in the present case, in the context of families.\par 
%
   %
   %
   %
To begin with, introduce the idea of a  geometrical quantity $\Theta_M$ on $M$. Informally speaking, symbol  $\Theta$ denotes the  {\it kind}  of the quantity $\Theta_M$ (which may be a function, a map, or a section of a vector bundle), while index \virg{$M$} is a remainder of the fact that our quantity is associated with $M$ (i.e., it is a function on $M$, a map from $M$ to another manifold, or a vector bundle over $M$). Let now $\g$ be a point of an auxiliary manifold $\G$, henceforth called the {\it manifold of parameters}. A set $\{\Theta_\g\}_{\g\in\G}$   of geometrical quantities of kind $\Theta$ on the manifold $M$ is what is usually referred to as a {\it family} of geometrical quantities (of   kind  $\Theta$) on $M$. However, even from a mere notational point of view, $\{\Theta_\g\}_{\g\in\G}$ is not an happy choice, since the  symbol $\g$, which stands for  a point of an extra manifold, is attached to the symbol $\Theta$, which denotes a geometrical quantity on $M$. Conceptually,  such a notation  immediately reveals its limits, since it is not even able to clarify the relationship between the smoothness of the whole family and the smoothness of each of its member.  So, the first aim of this paper is to   replace the naive idea
\begin{equation}\label{equazioneBASE}
\textrm{geometrical quantity}\ \Theta_M \longmapsto\textrm{family of geometrical quantities}\  \{\Theta_\g\}
\end{equation}
with a more conceptual one
\begin{eqnarray}\label{equazioneBASEconc}
\textrm{geometrical quantity}\ \Theta_M &\longmapsto&\textrm{geometrical quantity on the}   \Theta\\ && \textrm{{\it cylinder}}\ M\times\G\ \textrm{of the same kind as}\ \Theta_M.\nonumber
\end{eqnarray}
The reader should keep in mind that, throughout this paper, we retain the name {\it cylinder} for the cartesian product    $M\times\G$, since one of the most common  instances of families is obtained when the manifold of parameter is $\R$, or an interval in it.  \par
In order to work in the logic of differential calculus over commutative algebras, the idea \REF{equazioneBASEconc} must be adopted as the most  fundamental one, since it express the idea of a family in terms of an {\it algebra extension} $C^\infty(M)\longmapsto C^\infty(M\times\G)$, while the  usual one  \REF{equazioneBASE} can be retained  for more descriptive purposes. This  change of perspective makes it straightforward     to express, in a natural and easy way,    such matters as smoothness, derivation (Section \ref{secDerFam}), integration with respect to a parameter (Sec. \ref{secIntFam}),   of families of geometrical quantities, and other relevant properties (like that of being analytic, algebraic, meromorphic, etc.) which are not investigated here. \par
The second aim of this paper is to introduce and to systematically study what we called the {\it sliceable} quantities on cylinders, whose appearance in the theory of families is     explained as follows. The passage  
 from   \REF{equazioneBASEconc} to \REF{equazioneBASE} inevitably  requires  the   {\it slicing maps} $\iota_\g:p\in M\to (p,\g)\in M\times\G$. This leads to think that   {\it any} geometrical quantity $\Theta$ on the cylinder can be \virg{sliced} into a family $\{\Theta_\g\}_\g$, where $\Theta_\g$ is obtained \virg{by applying} $\iota_\g$ to $\Theta$, and the meaning of \virg{applying} depends one the kind of $\Theta$ (for instance, if $\Theta$ is a map, then $\Theta_\g$ is its pull--back $\iota_\g^\ast(\Theta)$). Now one should notice that,  first, not all quantities on the cylinder can be sliced and, second, that the \virg{slicing operation} may have a nontrivial kernel. The first phenomenon is typical of controvariant quantities (like vector fields), while the second concerns covariant quantities (to which Section  \ref{subsubCovHorQuant} is dedicated). Hence, we shall call {\it sliceable} those quantities $\Theta$ on the cylinder for which it makes sense to consider the restriction $\Theta|_{M\times\{\g\}}$ to the image of $\iota_\g$ and such that the correspondence
\begin{equation}\label{eqCorrPrincip}
\textrm{sliceable quantity }\Theta\textrm{ on }M\times\G \longleftrightarrow\textrm{family }\{\Theta|_{M\times\{\g\}}\}\textrm{ on }M
\end{equation}
becomes one--to--one. Unlike functions and maps, which are all sliceable, sliceable vector fields  will be understood as the sub--functor of \textit{vertical} derivations (Section  \ref{subsecVerticalVectorFields}) and, dually, sliceable differential forms will be their annihilator and, as such, called \textit{horizontal} (Sections  \ref{seubsubFormeDiffOrizzontaliOrdineUno} and \ref{secFamDiffForm}).  This confirms  the key role played by the  geometry of cylinders  in the conceptual study of families of quantities. \par 
Finally, since the  correspondence \REF{eqCorrPrincip}   suggests that   \virg{anything to which $\iota_\g$ can be applied} can be considered as a family,  then $\iota_\g$ may be also understood as the bundle pull--back $\iota_\g^\circ$. This leads to the notion of a {\it family along a map} (Section  \ref{secCylSlicFamGeo}), which, among other things, allows to give the conceptual definition of an  {\it infinitesimal homotopy}.
 \section{Cylinders  and families of geometrical quantities}\label{secCylSlicFamGeo}
 In this section we collect basic notations and definitions.\par
A product $M\times \G $ is a \emph{cylinder} over $M$.
  An element  $\g\in \G $ is called a \emph{parameter}. 
Map $\iota_\g:M\longrightarrow M\times \G $, $M\ni p \longmapsto  (p,\g)\in M\times \G $ is   the     \emph{slicing map} associated with $\g$, and $M_\g\df \iota_\g(M)=M\times\{\g\}$ is the \emph{slice} of parameter $\g$.
 Obviously, $\iota_\g$ is a smooth embedding. Its restriction, still denoted by $\iota_\g$, is a diffeomorphism between $M$ and $M_\g$.\par
Let $\Theta$ be a geometrical quantity on $M$ such that it makes sense to consider its restriction $\Theta_\g\df\Theta|_{M_\g}$, for all $\g\in\G$. Then the set  $\{\Theta_\g\}_{\g\in \G }$ is made of geometrical objects of  the same kind as $\Theta$. 
\begin{remark}\label{remStupidoMaChePoiServe}
 It should be stressed that  each element of $\{\Theta_g\}_{g\in \G }$ lives on a different manifold, namely, the slice $M_g$. Nonetheless,     $\iota_g$ allows to  transport   $\Theta_g$ back to $M$. This way, an object $\Theta_g$ on $M$ is obtained. 
 \end{remark}
\begin{definition}\label{defFamObj}
A set $\{\theta_\g\}_{\g\in \G }$ is called a (smooth) \emph{family of geometrical quantities} if there exists a quantity $\Theta$ on $M\times \G $ such that $\theta_\g$ corresponds to $\Theta_\g$ via $\iota_\g$, for any $\g\in \G $. 
\end{definition} 
%
%
%
The advantage of Definition \ref{defFamObj} is that a  family $\{\Theta_\g\}_{\g\in \G }$ is smooth by default, allowing us to skip the   modifier \virg{smooth} in the sequel. 
Depending on the kind of $\Theta$,   Defintion \ref{defFamObj} can be specialized as follows. A  \emph{family of   functions} on $M$, parametrized by $\G $,   is a function $f\in C^\infty(M\times \G )$.  
Similarly, a  \emph{family of maps} from $M$ to $N$, parametrized by $\G $, is a smooth map $F:M\times \G \to N$. \par
%
Definition \ref{defFamObj} roughly says  that a family is obtained by \virg{slicing something which lives on the cylinder} (by means of the $\iota_\g$'s). So,  Definition \ref{defFamObj} can be extended if one  introduces more   general objects \virg{which can be sliced}. Such objects may be  sections of an induced vector bundle $F^\circ(\pi)$, where $F:M\times \G \to N$ is   a smooth map and $\pi$ is a vector bundle on $N$.
\begin{definition}\label{defDefinizioniFamiglieDiSezioniLungoUnaMappa}
An element of the $C^\infty(M\times \G )$--module $\Gamma\left(F^\circ(\pi)\right)$ is called \emph{a  family of sections} of $\pi$ parametrized by $\G $ \emph{along} $F$ (or just a \emph{a  family of sections} of $\pi$ parametrized by $\G $ when   $M=N$ and $F=\pi_M$).
\end{definition}
\begin{remark}\label{remCheInveceServe}
In fact,  {\it any} element   $\sigma\in\Gamma\left(F^\circ(\pi)\right)$ can be {\it sliced} into a family $\{\sigma_\g\}_{\g\in \G }$, where $\sigma_\g\df \iota_\g^\circ(\sigma)$, and assignment  $\sigma\longmapsto\{\sigma_g\}_{g\in \G }$ is one--to--one. Accordingly,  
we  call $\sigma$ {\it sliceable} but, as it turns out, not all geometrical quantities on the cylinder will be  sliceable.
\end{remark}
 It is worth observing that Definition \ref{defFamObj} is not a particular case of Definition \ref{defDefinizioniFamiglieDiSezioniLungoUnaMappa}, since families of maps cannot be seen as elements of a module.
%
%
%
\begin{remark}\label{remFamSezDiFamFibr}
A section $\sigma\in\Gamma\left(\pi_M^\circ(\pi)\right)$  can be  {sliced} into a   family of sections $\{\sigma_\g\}_{\g\in \G }\subseteq\Gamma( \pi)$, since $\pi_M\circ \iota_\g=\id_M$ for all $\g$'s.  In general,   if   $\sigma$ is a section of $F^\circ(\pi)$, then   $\sigma_\g$ is \emph{not}   a section of $\pi$, but a section of $\pi_\g\df F_\g^\circ(\pi)$ instead, i.e., $\sigma_\g\in\Gamma(\pi_\g)$, $\g\in\G$. Informally speaking, $\sigma$ defines a family $\{\sigma_\g\}_{\g\in \G }$   of sections of a family of vector bundles $\{\pi_\g\}_{\g\in \G }$.
\end{remark}
%
It is worth observing that $\Gamma(\pi)\subseteq\Gamma\left(\pi_M^\circ(\pi)\right)$ via the map $\sigma\mapsto 1_{C^\infty(M\times \G )}\otimes\sigma$. The image of this embedding is constituted of sections of $\pi_M^\circ(\pi)$ that do \emph{not} depend on the extra parameter, and, as such, may be referred to as \emph{constant}.

\section{Derivatives of families}\label{secDerFam}
Due to their straightforwardness, all proofs in this sections will be omitted. 
We also assume that the reader knows about vertical derivations,   derivations along a map, related derivations, and   the theory of smooth envelopes (see \cite{Nestruev}  for more details). In the sequel, both $C^\infty(M)$ and  $C^\infty(\G )$ are naturally understood as subalgebras of $C^\infty(M\times \G )$ via the canonical monomorphisms $\pi_M^\ast$ and $\pi_\G ^\ast$, respectively, and $C^\infty(M)\otimes_\R C^\infty(\G )$   as a subalgebra of $C^\infty(M\times \G )$, via the product map $\pi_M^\ast\otimes\pi_\G ^\ast$.
  $C^\infty(M\times \G )$ is  tacitly understood both as a  $C^\infty(M)$-- and as a $C^\infty(\G )$--module.\par%
We show how the peculiar geometry of $M\times \G $ allows to lift     any vector field on $\G $   to a $\pi_M$--vertical vector field. In its turn, such  a lift allows to give an intrinsic definition of {\it derivative} of a family.\par Let $P$ be a $C^\infty(M\times \G )$--module.  
\begin{lemma}\label{lemEstDerivCylind}
Given  a $P$--valued derivation   $X$ (resp., $Y$)   of $C^\infty(M)$ (resp., $C^\infty(\G )$), there exists a   unique  $P$--valued derivation $Z$ of $C^\infty(M\times \G )$, simultaneously extending $X$ and $Y$.       
\end{lemma}
%
\begin{lemma}
Given vector fields $X\in D(M)$ and $Y\in D(\G )$, a unique vector field $Z\in D(M\times \G )$ exists, such that 
\begin{eqnarray*}
Z\circ \pi_M^\ast&=&\pi_M^\ast\circ X,\\
Z\circ \pi_\G ^\ast&=&\pi_\G ^\ast\circ Y.
\end{eqnarray*}
\end{lemma}

The last two conditions mean that $Z$ is $\pi_M$--related to $X$ and $\pi_\G $--related to $Y$. Observe that if $X=0$ (resp., $Y=0$) then $Z$ is $\pi_M$--vertical (resp., $\pi_\G $--vertical). In other words, it holds the following Corollary.

\begin{corollary}\label{corCanExtVectFieldCylind}
Any vector field $X\in D(M)$ (resp., $Y\in D(\G )$) can be lifted to an unique $\pi_\G $--vertical vector field $\widetilde X$ (resp., $\pi_M$--vertical vector field $\widetilde Y$) of $M\times \G $. 
\end{corollary}

Vector field $\widetilde X$ (resp., $\widetilde Y$) above is   the \emph{canonical lifting} of $X$ (resp., $Y$).

\begin{remark}[Coordinates]\label{exCoorLiftDueVett}
Let $\{x^1,\ldots,x^n\}$ be local coordinates on $M$ and let $\{y^1,\ldots,y^m\}$ be local coordinates on $\G $. Then the lifting $Z$ of $X=X^i\frac{\partial}{\partial x^i}$ and $Y=Y^j\frac{\partial}{\partial y^j}$ (see Lemma \ref{lemEstDerivCylind}) is given by $Z=\pi_M^\ast(X^i)\frac{\partial}{\partial x^i}+\pi_\G ^\ast(Y^j)\frac{\partial}{\partial y^j}$.
\end{remark}

Since  $\Gamma\left(\pi_M^\circ(\pi)\right)=C^\infty(M\times \G )\otimes_{C^\infty(M)}\Gamma(\pi)$, we also have the next Proposition.
\begin{proposition}\label{propTrivConn1}
  $\nabla_X\df\widetilde X\otimes\id$   is a well--defined first--order differential operator on $\Gamma\left(\pi_M^\circ(\pi)\right)$.
\end{proposition}
 
More precisely,  operator $\nabla_X$ from Proposition \ref{propTrivConn1} is  a \emph{der--operator} (see \cite{DeParis}) in $\Gamma\left(\pi_M^\circ(\pi)\right)$ over $\widetilde X$. It   allows   to extend the notion of derivative to smooth families.
\begin{definition}
Given $\sigma\in\Gamma\left(\pi_M^\circ(\pi)\right)$ and $X\in D(\G )$, the smooth family $\nabla_X(\sigma)$ is called the \emph{derivative} of $\sigma$ with respect to $X$.
\end{definition}
%


Let     $F:M\times\G\to N$ be a family of maps, and suppose that   $\g$ is  running along the trajectory of a vector field $X\in D(\G)$. Then   the $F_\g$'s describe a  \virg{trajectory} in some \virg{space of maps}, i.e., what is usually called a    \emph{deformation}.\footnote{When   $\G=[0,1]$, $F$ is an homotopy.}  In the standard approach, one tries  to add some differentiable structure to this \virg{space of maps}, in order to make it   possible to compute the \virg{velocity} of the deformation.  Thanks to Definition \ref{defDefinizioniFamiglieDiSezioniLungoUnaMappa},  the idea of velocity of a deformation can be formalized algebraically, without even thinking about the \virg{space of maps}.\par
More precisely, consider the composition $X(F)\df\widetilde X\circ F^\ast$, which    is a vector field along $F$, i.e., a smooth section of the bundle $F^\circ(\tau_N)$ induced from the tangent bundle of $N$ by $F$.  According to Definition \ref{defDefinizioniFamiglieDiSezioniLungoUnaMappa}, $X(F)$ represents   a smooth family of vector fields   parametrized by $\G $  along $F$. Moreover, as anticipated by Remark \ref{remFamSezDiFamFibr},   the member   of the family $X(F)$ which corresponds to the parameter $\g$ is  the vector field along $F_\g$ given by
\begin{equation}
X(F)_\g=\iota_\g^\ast\circ\widetilde X\circ F^\ast.
\end{equation}
So, Definition \ref{defHomotopiaInfinitesimale} below is the right algebraic counterpart of the  \virg{velocity of deformation}. Indeed, $X(F)_\g$ associate with a point $p\in M$ the velocity $X(F)_\g|_p\in T_{F(p,\g)}N$ of the trajectory $\g\longmapsto F(p,\g)$ in $N$, where $\g$ is running along a trajectory of $X$.
\begin{definition}\label{defHomotopiaInfinitesimale}
The $F$--relative vector field $X(F)$  is   the \emph{derivative} of $F$ with respect to $X\in D(\G )$. 
\end{definition} 

In the case when   $\G \subset\R$ and $X=\frac{\dd}{\dd t}$, the 
derivative $F'\df\frac{\dd}{\dd t}(F)$ is called  the \emph{infinitesimal homotopy} associated with $F$, and 
symbols   $\frac{\dd}{\dd t}(F)_{t_0}$,  $\left.\frac{\dd F}{\dd t}\right|_{t=t_0}$ and $F'_{t_0}$ are  interchangeable.\par
It is worth noticing that Proposition \ref{propTrivConn1} does not hold if one consider, instead of $\Gamma(\pi_M^\circ(\pi))$,    arbitrary families of sections along   $F$, since the canonical lifting $\widetilde X$ needs \emph{not} to be $F$--vertical. So, derivation(s) with respect to parameter(s) is not intrinsically defined in such a case. As explained in Remark \ref{remFamSezDiFamFibr}, the reason   is that sections of a family corresponding to different values of the parameter  cannot be compared.

	\begin{example}[Flow of a vector field]\label{exFlussoDICampo}
Let  $X$ be a complete vector field on $M$, $\G=\R$, and     consider  the \emph{time vector field} on $M\times\R$, i.e., the   $\pi_M$--vertical vector field 
\begin{equation}
 \frac{\partial}{\partial t}\df\widetilde{\frac{\dd}{\dd t}} .
\end{equation}
 Then     there exists a unique   family of maps $A: M\times\R\to M$, such that $A'=A^\ast\circ X$ and $A_0=\id_M$, i.e., 
\begin{equation}\label{eqFlow1}
\frac{\partial}{\partial t}\circ A^\ast=A^\ast\circ X\quad\textrm{and}\quad A\circ\iota_0=\id_M.
\end{equation}
Family $A$ fulfilling \REF{eqFlow1} is called the \emph{flow} generated by $X$.
Each member $A_t$ is a   diffeomorphism of $M$.\par
\end{example}
Observe that the infinitesimal homotopy $A'$ determined by $A$ can be interpreted as a family $\{A'_t\}_{t\in\R}$ of vector fields, where $A_0'=X$, but, in general, $A_t$ is the relative vector field $A_t^*\circ X$.  This indicates  the possibility that   \emph{any} homotopy may be seen as the   \virg{flow} associated with a family of \emph{relative} vector fields $\{X_t\}_{t\in\R}$, and this analogy will lead, in a surprisingly straightforward way, directly to the proof of the Homotopy Formula (Section \ref{secHomFor}).

 \begin{remark}\label{exLiftFlow1}
If $A:M\times \G \to M$ is the flow of the vector field $X\in D(M)$, then 
\begin{equation}
(M\times N)\times \G \ni(x,y,t)\stackrel{\widetilde{A}}{\longrightarrow}(A(x,t),y)\in M\times N
\end{equation}
is the flow of the canonical lift $\widetilde{X}$.
\end{remark}

 \begin{remark}\label{exLiftCanCompImme}
The canonical lift $\widetilde X$ of $X\in D(M)$ (resp., $\widetilde Y$ of $Y\in D(\G )$) is $\iota_\g$--compatible with $X$ (resp., $\iota_p$--compatible with $Y$) for all $\g\in \G $ (resp.,  for all $p\in M$).
Geometrically, the fact that $\widetilde{X}$ is  $\iota_\g$--compatible with $X$, means that, for any $f\in C^\infty(M\times \G )$ and   $p\in M$, we have 
\begin{equation}
\widetilde{X}_{(p,\g)}(f)=X_p(f_\g), \quad \g\in \G ,
\end{equation}
i.e., the action of $\widetilde{X}$ on a family  of functions $f$ coincides with the action of   $X$ on each its member $f_\g$. 
 \end{remark}

 \section{Integration of   families}\label{secIntFam}
 In this Section  we will define the \virg{inverse} of derivative operation, i.e.,  integration, by  extending  a bounded functional $\F\in C^\infty(\G)^\vee$    to families of objects,   much as Corollary \ref{corCanExtVectFieldCylind} allows to define the derivative operation by lifting derivations of $\G$ to the cylinder. To this end,     both  $C^\infty(\G)$, with $\G$   compact, and  $ C_c^\infty(\G) $  are equipped with the norm of the maximum, so that   $\F$ is   \emph{bounded}    if     $|\F(\varphi)|\leq K\| \varphi\|$ for a given $K\geq 0$, and all $\varphi$'s.
\begin{remark}\label{exConvergenza}
If $\F$ is bounded and $\varphi_t$ converges to $\varphi$ point--wisely in $C^\infty(\G)$ as $t\to 0$, then   $\F(\varphi_t)\longrightarrow \F(\varphi)$ in $\R$ as $t\to 0$.\par
\end{remark}\par
Let $h_x\in\mathrm{Spec}_\R(C^\infty(M))$ (see \cite{Nestruev}) be the        evaluation map  at $x\in M$,    and $f\in C^\infty(M\times\G)$.  Then, regarding  $f$   as a family of functions on $\G$ parametrized by $M$, we can turn each its member $f_x=\iota_x^\ast(f)$ into the real number $\F(f_x)$. In other words, family $f$ is turned into the function   $\widetilde{\F}(f):x\longmapsto  {\F}(f_x)$.
\begin{proposition}
Let $\F$ be bounded,  $f\in C^\infty(M\times\G)$ and $\widetilde{\F}(f)(x)\df  {\F}(f_x)$. Then
 $\widetilde\F(f)\in C^\infty(M)$,  
\begin{equation}\label{eqDiagramaEstenzioneFunzionaleLimitato}
\xymatrix{
C^\infty(M\times\G)\ar[r]^{\widetilde\F}\ar[d]^{\iota_x^\ast}&C^\infty(M)\ar[d]^{h_x}\\
C^\infty(\G)\ar[r]^{\F}&\R}
\end{equation}
is a commutative diagram,   and $\widetilde\F$ is $C^\infty(M)$--linear.
\end{proposition}
 \begin{proof} Obviously, if one replaces $C^\infty(M)$ with $\R^M$, \REF{eqDiagramaEstenzioneFunzionaleLimitato}  becomes commutative, and 
 \begin{equation*}
\widetilde\F(hf)=\{\F(h(x)\iota_x^\ast(f))\}_{x\in M}=\{h(x)\F(\iota_x^\ast(f))\}_{x\in M}=h\widetilde\F(f),
\end{equation*}
  for $h\in C^\infty(M)$ and $f\in C^\infty(M\times\G)$,
due to $\R$--linearity of $\F$, so   $C^\infty(M)$--linearity holds too. It remains to be shown that  $\widetilde\F(f)\in C^\infty(M)$.\par To this end, consider the flow $\{A_t\}$ generated by a vector field $X\in D(M)$. Since $(h_x\circ A_t^\ast)(\varphi)=A_t^\ast(\varphi)(x)=\varphi(A_t(x))$,  $\varphi\in \R^M$, then $h_x\circ A_t^\ast=h_{A_t(x)}$. Therefore,
\begin{equation}\label{eqIntSmFam1}
A_t^\ast(\widetilde\F(f))(x)=\left(h_x\circ A_t^\ast\right)(\widetilde\F(f))=\left(h_{A_t(x)}\circ\widetilde\F\right)(f)=\left(\F\circ\iota^\ast_{A_t(x)}\right)(f).
\end{equation}
Let $\widetilde X\in D(M\times\G)$ be the canonical lift  of $X$  and $\{\widetilde A_t\}$  its flow (see Remark \ref{exLiftFlow1}). Then  $\widetilde A_t\circ\iota_x=\iota_{A_t(x)}$  and   the existence of the derivative $\left.\frac{\dd}{\dd t}\right|_{t=0}\widetilde A_t^\ast(f)$ implies  the existence of the derivative 
\begin{equation}\label{eqIntSmFam2}
\left.\frac{\dd}{\dd t}\right|_{t=0}(\F\circ\iota_x^\ast)\left(\widetilde A_t^\ast(f)\right)
\end{equation}
(see Remark \ref{exConvergenza}), which coincides with $(\F\circ\iota_x^\ast)(\widetilde X (f))=\left(\F(\widetilde X (f))\right)(x)$. 
On the other hand, thanks  to \REF{eqIntSmFam1}, 
$
(\F\circ\iota_x^\ast)\left(\widetilde A_t^\ast(f)\right)=(\F\circ\iota_x^\ast\circ\widetilde A_t^\ast)(f)=(\F\circ\iota_{A_t(x)}^\ast)(f)=A_t^\ast(\widetilde\F(f))(x)
$.
Define
\begin{equation}\label{eqIntSmFam4prima}
X(\widetilde\F(f))(x)\df\left.\frac{\dd}{\dd t}\right|_{t=0}A_t^\ast(\widetilde\F(f))(x),\quad  x\in M,
\end{equation}
and observe that  the   derivative in \REF{eqIntSmFam4prima}  is well--defined  because of   \REF{eqIntSmFam2}. So,  \REF{eqIntSmFam4prima} reads
\begin{equation}\label{eqIntSmFam4}
X(\widetilde\F(f))=\widetilde\F(\widetilde X (f)).
\end{equation}
It follows immediately from \REF{eqIntSmFam4} that the action of any differential operator $\Delta=X_1\circ\cdots\circ X_s$, $X_i\in D(M)$, is well--defined on $\widetilde\F(f)$ and
\begin{equation}
\Delta(\widetilde\F(f))=\widetilde\F(\widetilde X_1(\ldots\widetilde X_s(f)\ldots)).
\end{equation}
This proves smoothness of $\widetilde\F(f)$.\footnote{If a subset $A\subseteq\R^M$ is closed under the action of differential operators on $M$,  then $A\subseteq C^\infty(M)$.}
\end{proof}
%
Since $\widetilde\F$ coincides with $\F$ on the subalgebra $C^\infty( \G )$ of $C^\infty(M\times \G )$,  it is appropriate to call it the {\it lift} of $\F$.  
\begin{proposition}\label{propEstenzioneDellIntegraleAFAmiglieDISezioni}
Operator   $\widetilde\F$ extends to a $C^\infty(M)$--homomorphism $\widetilde\F:\Gamma\left(\pi_M^\circ(\pi)\right)\longrightarrow\Gamma(\pi)$, denoted by the same symbol.
\end{proposition}
\begin{proof} Due to $C^\infty(M)$--linearity of $\F$, the the product $ \widetilde\F\otimes\id$ is a well--defined $C^\infty(M)$--homomorphism from $C^\infty(M\times \G )\otimes_{C^\infty(M)}\Gamma(\pi)$ to $\Gamma(\pi)$.\end{proof}
If  $\G$ is compact and $\F=\int_\G$, then      the section 
\begin{equation*}
\widetilde{\int_\G}\sigma\in\Gamma(\pi),  
\end{equation*}
  is called the \emph{integral}  (over   $\G$) of the family of sections  $\sigma\in\Gamma\left(\pi_M^\circ(\pi)\right)$.
Corollary \ref{corIntSmFam} below  is a straightforward generalization   of an elementary property of differentiation, namely that it commutes with integration w.r.t.  other  parameter(s). 
\begin{corollary}\label{corIntSmFam}
Let $X:  C^\infty(M) \longrightarrow \Gamma(\pi)$ be a derivation and $\widetilde X$   its canonical lift. Then the diagram
\begin{equation}
\xymatrix{%
C^\infty(M)\ar[r]^{X}&\Gamma(\pi)\\
C^\infty(M\times \G )\ar[u]^{\widetilde\F}\ar[r]^{\widetilde X}&\Gamma\left(\pi_M^\circ(\pi)\right),\ar[u]^{\widetilde\F}%
}
\end{equation}
is commutative.
\end{corollary}
\begin{proof}
Locally, $X=  X_i\otimes s^i$, where $X_i\in D(M)$ and $\{s^i\}\subseteq\Gamma(\pi)$ is a local basis.  Then $\widetilde{X}= \widetilde{X}_i\otimes s^i$ and  $\widetilde{X}(f)= \widetilde{X}_i(f)\otimes s^i$,  for every  $f\in C^\infty(M\times \G )$, so that    $\widetilde\F(\widetilde{X}(f))= \widetilde\F(\widetilde{X}_i(f))\otimes s^i$ (see Proposition \ref{propEstenzioneDellIntegraleAFAmiglieDISezioni}). But,  in view of \REF{eqIntSmFam4}, $ \widetilde\F(\widetilde{X}_i(f))\otimes s^i$  coincides with $  X_i(\widetilde\F(f))\otimes s^i$, i.e.,   $X(\widetilde\F(f))$.
\end{proof}

Proposition \ref{propFormulaNL} below provides a sort of Newton--Leibniz formula depending on parameters running over $M$.  
\begin{proposition}\label{propFormulaNL}
If $F:M\times [a,b]\longrightarrow N$, then 
\begin{equation}\label{BellaFormula-1}
\widetilde{\int_a^b}\circ F'=F_b^\ast-F_a^\ast.
\end{equation}
\end{proposition}
\begin{proof}
 By evaluating both sides of \REF{BellaFormula-1} on $f\in C^\infty(N)$ and then on $x\in M$ we obtain
 \begin{equation}\label{BellaFormula-1bis}
\left(h_x\circ\widetilde{\int_a^b}\circ F'\right)(f)=f(F_b(x))-f(F_a(x)).
\end{equation}
 In view of Definition  \ref{defHomotopiaInfinitesimale} and of commutativity of \REF{eqDiagramaEstenzioneFunzionaleLimitato}, the left--hand side of \REF{BellaFormula-1bis} reads 
  \begin{equation}\label{BellaFormula-1tris}
\left({\int_a^b}\circ\iota_x^\ast\circ \frac{\partial}{\partial t}\right)(F^\ast(f)).
\end{equation}
In its turn (see Remark \ref{exLiftCanCompImme}), \REF{BellaFormula-1tris} reads
$
\left({\int_a^b}\circ \frac{\dd}{\dd t}\right)(\iota_x^\ast(F^\ast(f)))
$,
i.e., 
$
(\iota_x^\ast(F^\ast(f)))(b)-(\iota_x^\ast(F^\ast(f)))(a)
$, which is precisely the right--hand side of  \REF{BellaFormula-1bis}.
\end{proof}

An interesting case of \REF{BellaFormula-1} is obtained when   $N=M\times\I$ and $F=\id_{M\times\I}$. Indeed,  $F_t=\iota_t$, and \REF{BellaFormula-1}  becomes
\begin{equation}\label{BellaFormula}
\iota_b^\ast-\iota_a^\ast=\widetilde{\int_a^b}\circ \frac{\partial}{\partial t}.
\end{equation}
Formula \REF{BellaFormula}  is more general than \REF{BellaFormula-1}, in that the  family to be integrated does not need to belong to $\im F^\ast$. Both \REF{BellaFormula-1} and \REF{BellaFormula}  admits a generalization to   families of forms (see Section \ref{secFamDiffForm}), which is essential to prove the homotopy formula (see Section  \ref{secHomFor}).    It should be stressed  that  an analogous  generalization    to sections of a generic vector bundle $\pi$ along a map $F:M\times\I\to N$ is not always possible. Indeed, for a generic $\pi$, the \virg{derivative} $F'$ of $F$ is not defined, unless $\pi$ defines a  \virg{covariant quantity}, as explained in Section \ref{subsubCovHorQuant} below.

 \section{Families of covariant quantities}\label{subsubCovHorQuant} 
Roughly speaking, a {\it covariant quantity} on $M$ is a bundle $\pi_{\Phi,M}$ which is {\it naturally associated} with $M$. Naturality implies, in particular,   the existence of pull--backs and Lie derivatives of sections of  $\pi_{\Phi,M}$. The formers allow to define families of covariant quantities, and the latter to formalize their derivative.\par More precisely,  let  $\FF$  be a representable functor of differential calculus  over $M$, and  $\Phi(M)$ the  $C^\infty(M)$--module which represents it in the category of geometric   $C^\infty(M)$--modules. Observe that a vector bundle $\pi_{\Phi,M}$  exists, such that  $\Phi(M)=\Gamma(\pi_{\Phi,M})$.
\begin{definition}
$\Theta\in\Phi(M)$ is a \emph{covariant quantity} (of type $\Phi$) on $M$.
\end{definition}
Since $M\longmapsto\Phi(M)$ is a functor,   any $F:M\longrightarrow N$ determines the functorial pull--back   $F^\ast:\Phi(N)\longrightarrow \Phi(M)$. On the other hand, being $\Phi(N)$ a module of sections, there is also a bundle--theoretic pull--back $F^\circ:\Gamma( \pi_{\Phi,N})\longrightarrow  \Gamma(F^\circ(\pi_{\Phi,N}))$.
In general, however,   the module $\Gamma(F^\circ(\pi_{\Phi,N}))$ has  nothing in common with $\Gamma(\pi_{\Phi,M})=\Phi(M)$, being  neither a submodule of $\Gamma(\pi_{\Phi,M})$ nor an its quotient.

\begin{definition}\label{defDefinitioneElementiEffeORizzontali}
Elements of  the submodule   $\Phi_F(M)\subseteq\Phi(M)$ generated by $F^\ast(\Phi(N))$   are called \emph{$F$--horizontal}. 
\end{definition}
In other words,   $\Phi_F(M)$ is made of elements of $\Phi(M)$ which are  of the form $  f_i F^\ast(\omega^i)$, with $f_i\in C^\infty(M)$, $\omega^i\in\Phi(N)$. In particular,   $\Phi_{\pi_M}(M\times \G )$ is precisely the module of families of sections of $\pi_{\Phi, M}$ according to   Definition \ref{defDefinizioniFamiglieDiSezioniLungoUnaMappa}. Indeed, when $F=\pi_M$,  one has   $F^\circ=F^\ast$ and $\Gamma(\pi_M^\circ(\pi_{\Phi,M}))=\Phi_{\pi_M}(M\times \G)$. So, for   covariant quantities, the    vague idea of being sliceable  is properly formalized in terms of      $\pi_M$--horizontality.
\begin{definition}\label{defDefinizioneFamiglieOggettiCOvarianti}
Elements of $\Phi_{\pi_M}(M\times \G )$ are   \emph{families of $\Phi$--type quantities} on $M$.
\end{definition}
\begin{example}
The module $\Lambda^k(M)$ represents the functor $D_k:P\to D_k(P)$ (see \cite{Nestruev}).  By definition, $\FF=D_0$ is the identity functor and $\Phi(M)=\Lambda^0(M)=C^\infty(M)$. The   Newton--Leibniz formula  \REF{BellaFormula-1} for families of maps  corresponds, therefore, to $\Phi=\Lambda^0$, and it will be generalized   to $\Lambda^k$, $k>0$, in Section \ref{secHomFor}.\par
\end{example}
%
%
Definition \ref{defDefinizioneFamiglieOggettiCOvarianti} allows to differentiate families  of $\Phi$--type quantities on $M$ by means of Lie derivatives. More precisely, given a vector field  $X\in D(M)$ and its flow $A$     (see Example \ref{exFlussoDICampo}),  the Lie derivative $L_X^\Phi:\Phi(M)\longrightarrow\Phi(M)$ is the derivative 
\begin{equation}\label{eqDerLieCovQuan}
L_X^\Phi\df\left.\frac{\dd A_{t,\Phi}^\ast}{\dd t}\right|_{t=0}
\end{equation}
of the family of functorial pull--backs $A_{t,\Phi}^\ast: \Phi(N)\longrightarrow\Phi(N)$. 
When $\Phi=\Lambda$, the Cartan formula allows to define  $L_X^\Phi$ in a pure algebraic way,\footnote{The algebraic definition of the Lie derivative for arbitrary $\Phi$'s is a more delicate problem, which  do not touch here.} namely
\begin{equation*}
L_X^{\Lambda}\df i_X\circ d + d\circ i_X.
\end{equation*}

So,  the Lie derivative   $L_{\widetilde{X}}^\Phi$ act as well on the submodule $\Phi_{\pi_M}(M\times \G )\subseteq \Phi(M)$ and on the subspace $\pi_M^\ast((\Phi(M))\subseteq \Phi_{\pi_M}(M\times \G )$.  Corollary \ref{corNablaXLX} below shows that $L_{\widetilde X}^\Phi$ actually vanishes on $\pi_M^\ast((\Phi(M))$, while on $\Phi_{\pi_M}(M\times \G )\subseteq \Phi(M)$ it coincides with the derivative operator $\nabla_X$ defined in Proposition \ref{propTrivConn1}.
%
\begin{proposition}\label{propVertLX}
$L_{\widetilde{X}}^\Phi$  preserves $\Gamma(\pi_M^\circ(\xi))$ and is $\pi_M$--vertical.
\end{proposition}
\begin{proof}
Let $\widetilde{A}_t$ be the flow generated by $\widetilde{X}$. Then   $\pi_M\circ\widetilde{A}_t=\pi_M$ (see Corollary \ref{corCanExtVectFieldCylind})  and,  in view of \REF{eqDerLieCovQuan},
\begin{equation*}
L_{\widetilde{X}}^\Phi\circ\pi_M^\ast=\left.\frac{\dd {\widetilde{A}}_{t,\Phi}^\ast}{\dd t}\right|_{t=0}\circ\pi_M^\ast=\left.\frac{\dd {\widetilde{A}}_{t,\Phi}^\ast\circ\pi_M^\ast}{\dd t}\right|_{t=0}=\left.\frac{\pi_M^\ast}{\dd t}\right|_{t=0}=0.
\end{equation*}
So, $L_{\widetilde{X}}^\Phi$ is $\pi_M$--vertical. Let now $\rho=f_i\pi_M^\ast(\Theta^i)\in\Gamma(\pi_M^\circ(\xi))$, with   $f_i\in C^\infty(M\times \G )$ and $\Theta^i\in\Phi(M)$.   Then $
L_{\widetilde{X}}^\Phi(\rho)= X(f_i)\pi_M^\ast(\Theta^i)\in\Gamma(\pi_M^\circ(\xi))
$ follows from    Leibniz rule and $\pi_M$--verticality of $L_{\widetilde{X}}^\Phi$.  
\end{proof}
Proposition \ref{propVertLX} allows an immediate proof of Corollary \ref{corNablaXLX} below.
\begin{corollary}\label{corNablaXLX}\ \par
\begin{enumerate}
\item $L_{\widetilde X}^\Phi$ restricted to $\Gamma(\pi_M^\circ(\xi))$ coincides with $\nabla_X$,
\item  $\pi_M^\ast((\Phi(M))$ consists of constant families of $\Phi$--type covariant quantities.

\end{enumerate}
\end{corollary}
%
%
In the case of  $\Phi=\Lambda$,  Definition \ref{defDefinizioneFamiglieOggettiCOvarianti}  corresponds to  the well--known  notion of horizontal forms, i.e.,      families of forms according to Definition \ref{defDefinitioneElementiEffeORizzontali}. Details will be discussed in Section \ref{secFamDiffForm}. 
Section \ref{subsecVerticalVectorFields} below  focuses on the dual side, i.e.,  families of vector fields. 
 \section{Example of families of \virg{controvariant quantities}: vector fields and derivations}\label{subsecVerticalVectorFields}
Following paradigm \REF{equazioneBASEconc},  a family $\{Z_g\}$ of vector fields should correspond to an  element   $Z\in D(M\times \G )$.  But  $Z$ is a controvariant quantity, so the pull--back $\iota_\g^\ast$ cannot be applied to it. On the other hand, $Z$ may be interpreted as a section of the tangent bundle  (see Remark \ref{remCheInveceServe}), and as such  the bundle--theoretic pull--back can be applied to it. However, the so--obtained vector field  $\iota_\g^\circ (Z)$ is   a {\it relative} one.\par
In this Section we implement correspondence \REF{eqCorrPrincip} in the case of vector fields. The first step  is to put 
\begin{equation}\label{eqCampoAffettato}
Z_\g\df\iota_\g^*\circ Z\circ\pi_M^*.
\end{equation}
The reason of choice \REF{eqCampoAffettato} is that $Z_\g$ is precisely the restriction $Z|_{M\times\{\g\}}$ via identification $\iota_\g$ (see Remark \ref{remStupidoMaChePoiServe}). 
The second task is to determine which submodule should be replaced to $D(M\times \G )$ in order to make 
\begin{equation}\label{eqCorrispondezaCampiNonNecessariamenteBuoniAFamiglieDiCampi}
D(M\times \G )\ni Z\longmapsto \{Z_g\}_{g\in \G }
\end{equation}
 a bijection, i.e., to discover what is the right formalization of  a {\it sliceable vector field}. Following intuition, a vector field $Z$ is {\it sliceable} if it is tangent to all the slices $M_\g$'s, i.e.,  if it is $\pi_\G $--vertical. This motivates Definition \ref{defDefinizioneIngenuaDiFamiglieDiCampiDelProf} below.
%
 \begin{definition}\label{defDefinizioneIngenuaDiFamiglieDiCampiDelProf}
A $\G $--parametrized \emph{family of vector fields} on a manifold $M$ is a $\pi_\G $--vertical field on $M\times \G $.
\end{definition}
So, unlike a family of covariant quantities on $M$, which is a  $\pi_M$--horizontal quantity on $M\times\G$, a family of  such \virg{controvariant quantities} as vector fields on $M$,  is made of   $\pi_\G$--vertical quantities on $M\times\G$. Nonetheless,     Proposition \ref{propVertDer1} below shows that Definition \ref{defDefinizioneIngenuaDiFamiglieDiCampiDelProf} above---much as Definition \ref{defDefinizioneFamiglieOggettiCOvarianti} for covariant quantities---is but a particular cases of  Definition \ref{defDefinizioniFamiglieDiSezioniLungoUnaMappa}. \par
%
%
%
%
%
%
%
%
%
%
%
%
To this end, notice that   a $\pi_\G $-- (resp., $\pi_M$--)vertical vector field $Z$ on the cylinder $M\times \G $ is uniquely determined by its restriction $Z|_{C^\infty(M)}$ (resp., $Z|_{C^\infty(\G )}$) since, by definition, $Z$ vanishes on $C^\infty(\G )$ (resp., $C^\infty(M)$)  (see Lemma \ref{lemEstDerivCylind}).  But $Z|_{C^\infty(M)}$ is a $C^\infty(M\times \G )$--valued derivation of $C^\infty(M)$ (resp., $C^\infty(\G )$), i.e., a relative vector field along the map $\pi_M$ (resp., $\pi_\G $).
\begin{remark}
In the same coordinates as Remark \ref{exCoorLiftDueVett}, it is easy to see that $Z$ is a vector field along $\pi_M$ if and only if $Z=Z^i\frac{\partial}{\partial x^i}$, $Z^i\in C^\infty(M\times \G )$.
\end{remark} 
On the other hand, a vector field  $Z$ along $\pi_M$ is a section of the induced bundle $\pi_M^\circ(\tau_M)$ (see \cite{Nestruev}). Moreover, 
\begin{equation}\label{eqRelVecFld1}
D_{\pi_M}(M\times N)\df \Gamma (\pi_M^\circ(\tau_M)) 
\end{equation}
 is a sub--$C^\infty(M\times \G )$ of $D(M\times \G )$, naturally isomorphic to  $C^\infty(M\times N)\otimes_{C^\infty(M)}D(M)$.    Similarly if  $Z$ is a vector field along $\pi_\G $.  
 In other words, it  is natural to identify $\pi_\G $-- (resp., $\pi_M$--)vertical vector fields with vector fields along $\pi_M$ (resp., $\pi_\G $). Proposition \ref{propVertDer1} below, whose easy proof is omitted, shows the functoriality of such identification.
\begin{remark}
For any  $C^\infty(M\times \G )$--module $P$, define the submodule $D_{\pi_M}(P)\subseteq D(P)$ of $P$--valued derivations of  $C^\infty(M\times \G )$ along $\pi_M$. Then  the correspondence $P\longmapsto D_{\pi_M}(P)$  is a functor.
\end{remark}

\begin{proposition}\label{propVertDer1}
Functor $D_{\pi_M}$ (reps., $D_{\pi_\G }$) is naturally identified with functor of $\pi_\G $-- (resp., $\pi_M$--)vertical derivations.
\end{proposition}

Now that Definition \ref{defDefinizioneIngenuaDiFamiglieDiCampiDelProf} has become a particular case of Definition \ref{defDefinizioniFamiglieDiSezioniLungoUnaMappa}, it can be generalized to arbitrary differential operators.
\begin{definition}\label{defDefinizioneGENDelProf}
A $\G $--parametrized   \emph{family of differential operators} between  $C^\infty(M)$--modules  $P$ and $Q$ is a $C^\infty(\G )$--linear differential operator
\begin{equation}\label{eqDefFamOpDiff}
\Delta: C^\infty(M\times \G ) \otimes_{C^\infty(M)} P \longmapsto C^\infty(M\times \G ) \otimes_{C^\infty(M)} Q.
\end{equation}
\end{definition}


 %
 %
 %
%
If $P=\Gamma(\eta)$ and $Q=\Gamma(\xi)$, then operator   $\Delta$ from Definition \ref{defDefinizioneGENDelProf} can be naively interpreted as a a family   $\{\Delta_\g\}_{\g\in\G}$ of $\xi$--valued differential operator on $\eta$. Indeed, in this case \REF{eqDefFamOpDiff} reads 
\begin{equation*}
\Delta: \sigma\in\Gamma(\pi_M^\circ(\eta))\longmapsto \Delta(\sigma)\in \Gamma(\pi_M^\circ(\xi)),
\end{equation*}
i.e., $\Delta $ maps a  a family $\sigma$ of sections of $\eta$ into a family $\Delta(\sigma) $ of section of $\xi$, in such a way that $\Delta(\sigma)_\g=\Delta_\g(\sigma_\g)$, where $\Delta_g\in\Diff(\Gamma(\eta),\Gamma(\xi))$  (see   Definition \ref{defDefinizioniFamiglieDiSezioniLungoUnaMappa}). \par
%
%
\begin{example}
A $\G $--parametrized   \emph{family of derivations of the algebra $C^\infty(M)$ with values in a $C^\infty(M)$--module $P$} is a $C^\infty(\G )$--linear derivation $Z:C^\infty(M\times \G )\longrightarrow C^\infty(M\times \G )\otimes_{C^\infty(M)}P$.
%
By using  the same coordinates as Remark   \ref{exCoorLiftDueVett}, and   a local basis  $\{s^j\}$  of $P=\Gamma(\xi)$,  a family $Z$ of $P$--valued derivations can be represented as $Z=Z^i_j\otimes\frac{\partial}{\partial x^i}\otimes s^j$, with $Z_j^i\in C^\infty(M\times\G)$. Accordingly, $Z_\g=\iota_\g^\ast( Z^i_j)\frac{\partial}{\partial x^i}\otimes s^j$. 
\end{example}
%
%
%
%
%
%
%
%
\section{Vertical and horizontal differential forms}\label{seubsubFormeDiffOrizzontaliOrdineUno}
Proposition \ref{propVertDer1} above identifies the notion of a  {\it sliceable derivation} and, in particular, of a family of vector fields, with the functor $D_{\pi_M}$ of derivations along the canonical projection $\pi_M$. In this Section   we show that  $D_{\pi_M}$ is representable, and that its representative object is precisely the module of {\it horizontal} differential 1--forms, which, in its turn, according to Definition \ref{defDefinizioneFamiglieOggettiCOvarianti}, corresponds to families of differential 1--forms.\footnote{Indeed, differential 1--forms are special type of covariant quantities  (see Section \ref{subsubCovHorQuant}).}   In other words, families of differential forms can be thought of as the representative object of families of vector fields, in total agreement with  the logic of differential calculus.\par
Since families of vector fields correspond not only to the vector fields along $\pi_M$, but also to the  vertical $\pi_\G$--vector fields (Proposition \ref{propVertDer1}), it is natural to look for the representative object of vertical derivations. As   Lemma \ref{lemLemmaInutileMaCheCiVuolePerForza} below shows, such an object is the quotient of $\Lambda^1(M)$ w.r.t. the submodule of horizontal forms, in the sense of Definition \ref{defDefinitioneElementiEffeORizzontali}. However, when $\Phi=\Lambda^1$, Definition \ref{defDefinitioneElementiEffeORizzontali} gains an important geometrical meaning,  so it is worth specializing it here.\par
In order to have the most general definition, let $f:M\rightarrow N$ be a smooth map.
\begin{definition}\label{defUnoFormeOrizzontali}
The sub--module $\Lambda^1_f(M)$ of $\Lambda^1(M)$ generated by the image of $f^\ast$ is   the module of ($f$--)\emph{horizontal} 1--forms on $M$.
\end{definition}
Geometrically, a 1--form $\omega$ is {\it horizontal} when it is constant along the fibers of $f$, i.e.,   $i_X(\omega)=0 $ for all $f$--vertical vector fields $X\in D_f^v(M)$. Obviously,   $i_X(f^*(\eta))=0$, i.e.,   a 1--form which is horizontal in the sense of Definition \ref{defUnoFormeOrizzontali} is also horizontal in the geometrical sense.   Lemma \ref{lemLemmaInutileMaCheCiVuolePerForza} below shows that the converse holds as well.\par
Put 
\begin{equation}\label{eqUnoFormeVerticali}
\overline{\Lambda}_f^1(M)\df \frac{\Lambda^1(M)}{\Lambda_f^1(M)}.
\end{equation}
\begin{lemma}\label{lemLemmaInutileMaCheCiVuolePerForza}
 %
\begin{equation*}
D_f^v(M)\cong \Hom(\overline{\Lambda}_f^1(M), C^\infty(M)).
\end{equation*}
\end{lemma} 
\begin{proof}
In view of  the isomorphism 
\begin{equation}\label{eqEqDualCamForm}
D(M)\ni X\leftrightarrow i_X\in\Hom(\Lambda^1(M), C^\infty(M)),
\end{equation}
 a submodule  $Q\subseteq\Lambda^1(M)$ must exist, such that $D_f^v=\Ann(Q)$. But a vector field $X\in D(M)$ is $f$--vertical if and only if $i_X(\omega)=0$ for any  $\omega\in \Lambda_f^1(M)$,  so $Q=\Lambda^1_f(M)$.
\end{proof}
\begin{definition}\label{defUnoFormeVerticali}
An element  $[\omega]_{\Lambda_f^1(M)}$ of the quotient module \REF{eqUnoFormeVerticali} is called an  \hbox{($f$--)}\emph{vertical 1--form}, and is   denoted  by $\overline{\omega}$. 
\end{definition}
%
%
%
\begin{corollary}\label{exEsercizioInCuiSonoConvogliateTutteLeSeccatureFuntorialiSuFormeVerticali}
The functor $D_f^v$  is represented by the module of $f$--vertical 1--forms, and   the natural embedding $D_f^v\subseteq D$ corresponds to the canonical projection\linebreak   $\Lambda^1(M)\longrightarrow \frac{\Lambda^1(M)}{\Lambda_f^1(M)}$ of representative objects.
\end{corollary}
%
%
Consider now the   cylinder $M\times\G$. In this case, the peculiar geometry of the manifold $M\times\G$ allows to identify vertical forms with respect to one projection with horizontal forms with respect to the other one. Details are as follows.
\begin{lemma}  
\begin{equation}\label{eqLeibnizRuleLambdaOne}
\Lambda^1(M\times \G )=\left(\Lambda^1(M)\overline{\otimes}_\R C^\infty(\G )\right)\oplus\left(C^\infty(M)\overline{\otimes}_\R\Lambda^1(\G )\right),
\end{equation}
\end{lemma}
\begin{proof}See \cite{DeParis}.
\end{proof}
\begin{proposition}\label{propEqVertHorForms}
$\pi_M$-- (resp., $\pi_\G $--)vertical 1--forms on $M\times \G $ are identified with $\pi_\G $-- (resp., $\pi_M$-- )horizontal 1--forms.
\end{proposition}
\begin{proof}
To prove both assertions, it suffices to interpret  \REF{eqLeibnizRuleLambdaOne} as
\begin{equation}\label{eqSplitForms1}
\Lambda^1(M\times \G )=\Lambda_{\pi_M}^1(M\times \G )\oplus\Lambda_{\pi_\G }^1(M\times \G ).
\end{equation}
In fact,  \REF{eqSplitForms1} follows from \REF{eqLeibnizRuleLambdaOne} since   $\Lambda^1(M)\overline{\otimes}_\R C^\infty(\G )$ coincides with  $\Lambda^1(M) \otimes_{C^\infty(M)}  C^\infty(M\times \G )$,\footnote{See \cite{DeParis} concerning the smoothened tensor product.} which is the submodule of $\Lambda^1(M\times \G )$ generated by $\im\pi_M^*$. The same for $\Lambda_{\pi_\G }^1(M\times \G )$.
\end{proof}
Combining Proposition \ref{propEqVertHorForms} above with Proposition \ref{propVertDer1} and Corollary \ref{exEsercizioInCuiSonoConvogliateTutteLeSeccatureFuntorialiSuFormeVerticali}, we easily obtain the next Corollary \ref{corUnPoStupido}.
\begin{corollary}\label{corUnPoStupido}
Functor $D_{\pi_M}$ (resp., $D_{\pi_\G }$) is represented by  $\Lambda_{\pi_M}^1(M\times \G )$ (resp.  $\Lambda_{\pi_\G }^1(M\times \G )$) in the category of geometric $C^\infty(M\times \G )$--modules. 
\end{corollary}
  In other words, derivations along $\pi_M$ (resp.,  $\pi_\G$) may be called $\pi_M$-- (resp.,  $\pi_\G$--)  \emph{horizontal  derivations}, since they are represented by $\pi_M$-- (resp.,  $\pi_\G$--)horizontal differential 1--forms. This way, in total analogy with  covariant quantities, sliceable derivations coincide with the horizontal ones, and Corollary \ref{corUnPoStupido} reads as the \virg{horizontal version} of the duality \REF{eqEqDualCamForm} between derivations and 1--forms.\par
%
%
%
In order to define families of higher--order differential forms, it is convenient to denote  by $\Lambda^{p,q}(M\times \G )$ the submodule of $\Lambda^{p+q}(M\times \G )$ generated  by\linebreak $\pi_M^\ast(\Lambda^p(M))\otimes_{C^\infty(M\times \G )}\pi_\G ^\ast(\Lambda^q(\G ))$.
\begin{definition}
 Elements of $\Lambda^{p,q}(M\times \G )$ are the \emph{forms of type $(p,q)$} on the cylinder $M\times\G$.
\end{definition}
\begin{corollary} 
The direct sum decomposition
\begin{equation}\label{eqSplitForms2}
\Lambda^k(M\times \G)=\bigoplus_{i=0}^k\Lambda^{i,k-i}(M\times \G )
\end{equation}
holds.
\end{corollary}
\begin{proof}
Easy follows from \REF{eqSplitForms1}.
\end{proof}
%

%
%
 %
 %
%
%
Definition \ref{defFamKappaForme} below complies with the general definition of families of objects (see Definition \ref{defDefinizioniFamiglieDiSezioniLungoUnaMappa}).
\begin{definition}\label{defFamKappaForme}
 A $\G$--parametrized family of $k$--forms is an element of the module $\Lambda^{k}_{\pi_M}(M\times\G)=\Lambda^{k,0}(M\times\G)$.
\end{definition}
However, the geometrical content of Definition \ref{defFamKappaForme} is not self--evident, and we may look for alternative ways to define families of higher order differential forms. Geometrically, since  horizontal 1--forms are annihilated by vertical vector fields (see Lemma \ref{lemLemmaInutileMaCheCiVuolePerForza}),   horizontal $k$--forms may be defined as those that are annihilated by vertical $k$--multivector fields, namely,
\begin{equation}\label{eqDefIdealeIConEffe}
\mathcal{I}_f \df \{ \omega\in\Lambda^+(M)\ |\ \omega(X_1,\ldots,X_{k})=0\quad\forall X_i,\ldots,X_{k}\in D^v,\ k=\deg(\omega)\}.
\end{equation}
But it is also possible to generalize Definition \ref{defUnoFormeOrizzontali},   so that horizontal forms of positive degree are given by   the ideal $<f^*(\Lambda^+(N))>$ of $\Lambda(M)$ generated by the image via $f^*$ of positive degree forms $\Lambda^+(N)$ on $N$.\par
%
Obviously,  $<f^*(\Lambda^+(N))>\subseteq \mathcal{I}_f$ and $\mathcal{I}_f $ is a differential ideal of $\Lambda(M)$.
\begin{lemma}\label{exEsercizioDoppioSugliIdeali}
 If $f$ is regular, then $\mathcal{I}_f =<f^*(\Lambda^+(N))>$.
\end{lemma}
\begin{proof}
 Well--known result in the theory of distributions. 
\end{proof}
So, under regularity conditions for $f$ (always assumed in the sequel), the algebraic notion of horizontal forms is geometrically interpreted in the context of  distributions, thus motivating Definition \ref{defFormeOrizzontali} below.
%
%
\begin{definition}\label{defFormeOrizzontali}
  $<f^*(\Lambda^+(N))>$ is the \emph{ideal  of horizontal ($f$)--forms}.
\end{definition}
%
%
%
To unveil the geometrical content of Definition \ref{defFormeOrizzontali}, observe that   the  family 
\begin{equation}\label{eqFamDistrVert}
M\ni x \longmapsto \ker d_xf\subseteq T_xM
\end{equation}
of tangent subspaces is  a Fr\"obenius distribution, whose maximal integral submanifolds are the fibers of $f$. \REF{eqFamDistrVert} is  called the  \emph{$f$--vertical distribution on $M$} and  is denoted by $\mathcal{V}_f$. Then the annihilator $\mathcal{V}_f\Lambda (M)$ of $\mathcal{V}_f$ is the ideal $\II_F$ defined by \REF{eqDefIdealeIConEffe}, i.e., the ideal of horizontal $f$--forms. Observe that $D_f^v(M)$ is precisely the module of vector fields belonging to  $\mathcal{V}_f$, and $D_f^v(M)_x=\ker d_xf$.
%
%
%
%
%
%
%
%
\begin{lemma}\label{lemLemmaCheUsaEsercizioSimpaticoDiAlgebraMultilineare}
For any $x\in M$, the space of skew--symmetric $k$--multilinear forms on $D_f^v(M)_x$   identifies with     $(\mathcal{I}_f\cap\Lambda^k(M))_x$.
\end{lemma}
\begin{proof}
First recall that (see \cite{LANG})
\begin{equation}
\Ann (L^{\wedge k})=\{ \varphi \in (V^{\wedge k})^\vee\ |\ \varphi_{|L^{\wedge k}} = 0\}=(L^{\wedge k})^\vee.
\end{equation}
%
%
If $L=\ker d_xf$ and $V=T_xM$, then  $(V^{\wedge k})^\vee$ identifies with $(T_x^\ast(M))^k$ and $\Ann (L^{\wedge k})$ with $(\mathcal{I}_f\cap\Lambda^k(M))_x$. To conclude the proof, observe that $(L^{\wedge k})^\vee$ is the space of   skew--symmetric $k$--multilinear forms on $D_f^v(M)_x$.
\end{proof}
 Lemma \ref{lemLemmaCheUsaEsercizioSimpaticoDiAlgebraMultilineare} motivates Definition \ref{defFormeVerticaliQualsiasi} below.
 \begin{definition} \label{defFormeVerticaliQualsiasi}
  \emph{$f$--vertical differential forms}, usually denoted by $\overline{\omega}$, are elements of 
\begin{equation}\label{eqFormeVerticali}
\overline{\Lambda}_f(M)=\frac{\Lambda(M)}{\mathcal{I}_f}.
\end{equation}
\end{definition}
Indeed,  $\overline{\Lambda}^k_f(M)$ can be interpreted as the \virg{co--distribution}   on $M$
\begin{equation}\label{eqFamFormeVert}
M\ni x \longmapsto\overline{\Lambda}^k_f(M)_x=\frac{(T^\ast_xM)^{\wedge k}}{\Ann ((\ker d_xf)^{\wedge k})}=(D_f^v(M)_x^{\wedge k})^\vee,
\end{equation}
 dual to the $k$--th power of the vertical distribution  \REF{eqFamDistrVert}.
%
Corollary \ref{corollarioDelCazzo} below is the  \virg{global analog} of  Lemma \ref{lemLemmaCheUsaEsercizioSimpaticoDiAlgebraMultilineare}.
\begin{corollary}\label{corollarioDelCazzo}
 An element $\overline{\omega}\in\overline{\Lambda}_f(M)$, $\omega\in\Lambda^k(M)$,  is naturally interpreted as the $k$--multilinear skew--symmetric    $C^\infty(M)$--multilinear map 
\begin{equation}
\underset{k\textrm{ times}}{\underbrace{D^v(M)\times\cdots\times D^v(M)}}\ni ( X_1,\ldots,X_k )\longmapsto \omega(X_1,\ldots, X_k)\in C^\infty(M).
\end{equation}
\end{corollary}
Definition \ref{defFormeVerticaliQualsiasi} makes it clear  that  $\overline{\Lambda}_f(M)$ is not only the dual to the module of vertical $k$--multivectors, but it   also inherits the quotient differential algebra structure from   $ {\Lambda(M)}$.  In its turn, such a  differential algebra structure will be used in Section \ref{secFamDiffForm} below to define a differential algebra structure on   families of forms. 
%
%
%
%
%
%
%
\section{Families of differential forms and their natural operations}\label{secFamDiffForm}
Among all geometrical quantities, differential forms are perhaps the most interesting one, due to the rich structure they possess. This is reflected by the variety of equivalent ways in which families of differential forms can be defined.     Namely,   a $\G$--parametrized family of differential forms is
\begin{itemize}
\item an element of $\Gamma(\pi_M^\circ(\xi))$, where  $\xi=\bigoplus_k(\tau^\ast_M)^{\wedge k}$ (Definition \ref{defDefinizioniFamiglieDiSezioniLungoUnaMappa}),
\item a $\pi_M$--horizontal $\Lambda$--type covariant quantity (Definition \ref{defDefinizioneFamiglieOggettiCOvarianti}),
\item an element of the direct sum ${\displaystyle\bigoplus_k} \Lambda^{k,0}(M\times\G)$ (Definition \ref{defFamKappaForme}).
\end{itemize}
Independently on the definition, the symbol $\Lambda_{\pi_M}(M\times\G)$  will be used  for the module of families of differential forms. It should be noticed that none of the  definitions above shows that  $\Lambda_{\pi_M}(M)$ possesses a differential algebra structure---so yet another perspective is needed.\par
To this end, it is enough to notice that the kernel of the correspondence \REF{equazioneBASEconc} which turns {\it any} form into a family coincides with the ideal  $\II_{\pi_\G}$ of $\pi_\G$--horizontal forms. In other words,  a $\G$--parametrized family of differential forms is also
\begin{itemize}
\item an element of the {\it differential algebra} $\overline{\Lambda}_{\pi_\G}(M\times\G)$ of $\pi_\G$--vertical differential forms (see Definition \ref{defFormeVerticaliQualsiasi}).\footnote{The same correspondence between families of quantities and vertical quantities was found in the context of vector fields (see Proposition \ref{propVertDer1}).}
\end{itemize}
The drawback of f the last definition is that, unlike the first three ones, it does not make $\Lambda_{\pi_M}(M\times\G)$ a submodule of $\Lambda (M\times\G)$, but rather an its quotient. Decomposition \REF{eqSplitForms2} allows to treat the last definition on the same footing as the others, thus obtaining a submodule of $\Lambda (M\times\G)$ which is also a differential algebra. More precisely,  
 introduce the canonical projections $p_{i,k-i}:\Lambda^k(M\times \G )\longrightarrow \Lambda^{i,k-1}(M\times \G )$. 
\begin{definition}\label{defHorizOp}
The $M$--\emph{horizontalization} of $k$--forms is   $p_{k,0}:\omega\mapsto\overline\omega\df p_{k,0}(\omega)$, while the  $M$--\emph{horizontalization} is   $p_{0} \df \bigoplus p_{k,0}$.
\end{definition}
%
Put $\overline d \df \left. p_{k,0}\circ d \right|_{\Lambda_{\pi_M}^k(M\times \G )}:\Lambda_{\pi_M}^k(M\times \G )\longrightarrow \Lambda_{\pi_M}^{k+1}(M\times \G )$. Then,  for  any $\omega\in\Lambda_{\pi_M}^k(M\times \G )$, $d^2\omega$ is exactly the $(k+2,0)$--component of the form $d^2\omega$. So, $\overline{d}^2=0$.  
Notice also that the differential $\overline{d}$ on $\Lambda_{\pi_M}(M\times \G )$ is precisely the one induced from the    differential of  $\overline{\Lambda}_{\pi_\G}(M\times\G)$ via the   identification
\begin{equation*}
\overline{\Lambda}_{\pi_\G}(M\times\G)\cong \Lambda_{\pi_M}(M\times \G ),
\end{equation*}
due to decomposition  \REF{eqSplitForms2}. So, a family of differential forms is also
\begin{itemize}
\item an element of  the differential algebra $\left(\Lambda_{\pi_M}(M\times \G ),\overline{d}\right)$.
\end{itemize}
The last point of view is the most complete and useful one, thus motivating Definition \ref{defHorizCompl} below. 
\begin{definition}\label{defHorizCompl}
   $\left(\Lambda_{\pi_M}(M\times \G ), \overline{d}\right)$ is the   \emph{$\pi_M$--horizontal de Rham complex}, and $\overline{d}$ is the {\it horizontal differential}.\end{definition}
%





%


Definition \ref{defHorizCompl} allows to formalize the well--known interchangeability of derivative and integration in the context of families of forms. Indeed,  integration of families (see Section \ref{secIntFam})    becomes more interesting when it is performed on differential forms, since it interacts with the horizontal differential. More precisely, as Corollary \ref{corHorizOp} below shows,   if $\G$ to be compact and $\F\in C^\infty(\G)^\vee$ is   bounded, then the lift $\widetilde\F$ is a $\Lambda(M)$--linear cochain map.\par
To this end, regard  the 1--st de Rham differential $d:C^\infty(M)\longrightarrow\Lambda^1(M)$ as a $\Lambda^1_{\pi_M}(M)$--valued map, and construct its $\pi_\G$--vertical  lift $\widetilde{d}$ as in   Proposition \ref{lemEstDerivCylind}.
\begin{proposition}
  $\widetilde{d}$ coincides with $\overline{d}$ and the   diagram
\begin{equation}\label{eqDiffCommInt1}
\xymatrix{%
C^\infty(M)\ar[r]^{d}&\Lambda^1(M)\\
C^\infty(M\times \G )\ar[u]^{\widetilde\F}\ar[r]^{\widetilde d}&\Lambda_{\pi_M}^1(M\times \G )\ar[u]^{\widetilde\F}%
}
\end{equation}
is commutative.
\end{proposition}
\begin{proof} Enough to show   that   $\overline{d}$ fulfills the properties of  the canonical lift of $d$ (see Proposition \ref{lemEstDerivCylind}). Indeed, if $f\in C^\infty(M)$, then $\overline{d}(\pi_M^\ast(f))=\overline{d\pi_M^\ast(f)}=\pi_M^\ast(df)$. So  $\overline{d}$ extends $d$. Moreover, if $g\in C^\infty(\G )$, then $\overline{d}(\pi_\G ^\ast(g))=0$, since $\pi_\G ^\ast(dg)\in\Lambda^{0,1}(M\times \G )$. So, $\overline{d}$ is $\pi_\G $--vertical.
\end{proof}
 \begin{corollary}\label{corHorizOp}
Operator $\widetilde\F$ is a  cochain map from $\Lambda_{\pi_M}(M\times \G )$ to $\Lambda(M)$,  of degree 0 and $\Lambda(M)$--linear.
\end{corollary}
\begin{proof}
A   horizontal $k$--form $\omega$ can be presented  as $\omega= f_i\pi_M^\ast(\eta^i)$, with $\eta_i\in\Lambda^k(M)$, and it is identified with $ f_i\otimes\eta^i$. If $\eta\in\Lambda(M)$, then $\omega\wedge\eta$ is identified with $ f_i\otimes(\eta^i\wedge\eta)$. So, $\widetilde\F(\omega\wedge\eta)=\F(f_i)\eta^i\wedge\eta=\left(\widetilde\F(\omega)\right)\wedge\eta$. This proves $\Lambda(M)$--linearity of $\widetilde{\F}$.\par
Moreover, $\overline{d}\omega=\overline{df_i\wedge \pi_M^\ast(\eta^i)+f_i \pi_M^\ast(d\eta^i)}=\overline{d}f\wedge \pi_M^\ast(\eta^i)+f_i \pi_M^\ast(d\eta^i)$ and  $\widetilde\F(\overline{d}\omega)=\left(\widetilde\F\circ\overline{d}\right)(f_i)\wedge \eta^i+\F(f_i)(d\eta^i)$. By commutativity of \REF{eqDiffCommInt1}, the last expression equals to $d(\F(f_i))\wedge\eta^i+\F(f_i)(d\eta^i)=d\left(\F(f_i)\eta^i\right)=d\left(\widetilde\F(\omega)\right)$.
\end{proof}
\begin{remark}
When $\F=\int_\G$, the   identity $\widetilde\F(\overline{d}\omega)=d\left(\widetilde\F(\omega)\right)$ gives the  familiar rule
\begin{equation}
\int_\G(d\omega_g)=d\left(\int_\G\omega_g\right)
\end{equation}
for \virg{taking the differential out of the  sign of integration}.
\end{remark}

 \section{The homotopy formula}\label{secHomFor}
%
The departing point to develop differential topology, which is the theory of  algebraic topology based on differential forms, is the so--called \emph{homotopy formula} \REF{eqFurmulaOmotopia}, which allows to  transform a homotopy connecting two maps from $M$ to $N$ into a chain homotopy connecting the induced maps  from $\Lambda(N)$ to $\Lambda(M)$.\footnote{It allows to prove  homotopy invariance of the de Rham cohomology, the most fundamental property of the de Rham cohomology on which all efficient computational algorithms are based. Another basic instrument of computing  de Rham cohomology, the \emph{suspension theorem}, needs a compact--supported version of the theory of families, not discussed here.}
Thanks to the theory of families of forms developed so far, we are able to show that the homotopy formula is a simple and natural consequence of the Cartan formula for the Lie derivative and the form--valued Newton--Leibniz formula \REF{BellissimaFormula}. Our proof differs from the classical ones which can be found in literature (see the classical book \cite{BottTu}) in that, being   purely algebraic, there is no need to check analytically the correctness of all the steps, thus focusing only its conceptual aspects. \par
Throughout this section, $F:M\times \G\to N$ is a smooth homotopy, $\omega$ is a form on $N$, and $\eta=F^*(\omega)$. So, $\eta$ determines the horizontal form $\overline{\eta}$, which in its turn can be regarded as a $\G$--parametrized family of forms  $\{\eta_\g\}$ (see Section \ref{secFamDiffForm} above) . To distinguish the case when $\G$ is $\R$, or an interval, we simply use the index $t$ instead of $\g$.\par
 The derivative of the family $\{\eta_\g\}$ along   a vector field $X\in D( \G )$   is  defined by means of the Lie derivative $L_{\widetilde{X}}$      (see Corollary \ref{corNablaXLX}). Denote by  $\{X(\eta_\g)\}$ the family of forms corresponding to $L_{\widetilde{X}}(\overline{\eta})$.\par
A key remark is that the family $\{X(\eta_\g)\}$ is also obtained by slicing the form $L_{X(F)}(\omega)$, where  ${\widetilde X}$  is the Lie derivative along the $F$--relative vector field $X(F)$.
When   $X=\frac{\dd}{\dd t}$ we just write $\frac{\dd F}{\dd t}$ instead of $X(F)$, and  $\eta'_t$ instead of $X(\eta_t)$.

%
%
To begin with, prove the form--valued Newton--Leibniz formula. 
 Put for simplicity    $\III_a^b=\int_a^b$ and $\I=[a,b]$.
 \begin{lemma}\label{lemLemmaCheNonEraUnLemmaMaLoEDiventato}
\begin{equation}\label{eqN-LFormuFormALG}
{\III_a^b}\circ \nabla_{\frac{\partial}{\partial t}}\circ p_0 =\iota_b^\ast-\iota_a^\ast.
\end{equation}
\end{lemma}
\begin{proof}
By using decomposition \REF{eqSplitForms2} represent   $\omega$ in the form
\begin{equation}
\omega=\sum_if_i\pi_M^\ast(\omega^i)+\rho\wedge\pi_\I^\ast(dt),\quad \omega^i\in \Lambda(M).
\end{equation}
%
Then, since $\overline{ \rho\wedge\pi_\I^\ast(dt)}=0$, we have
\begin{eqnarray*}
\left({\III_a^b}\circ\nabla_\frac{\dd}{\dd t}\circ p_0\right)(\omega)&=&{\III_a^b}\left(\nabla_\frac{\dd}{\dd t}\left(\overline{\sum_if_i\pi_M^\ast(\omega^i)+\rho\wedge\pi_\I^\ast(dt)}\right)\right)\\
={\III_a^b}\left(\nabla_\frac{\dd}{\dd t}\left(\sum_if_i\pi_M^\ast(\omega^i)\right)\right)&=&{\III_a^b}\left(\frac{\partial f_i}{\partial t}\pi_M^\ast(\omega^i)\right)\\
=\sum_i\left({\III_a^b}\frac{\partial f_i}{\partial t}\right)\omega^i&=&\sum_i\left[\left(\iota_b^\ast-\iota_a^\ast\right)(f_i)\right]\omega^i\\
&=&\sum_i\left(\iota_b^\ast(f_i)\omega^i-\iota_a^\ast(f_i)\omega^i\right).\\
\end{eqnarray*}
On the other hand, 
\begin{eqnarray*}
(\iota_b^\ast-\iota_a^\ast)(\omega)&= &(\iota_b^\ast-\iota_a^\ast)\left(\sum_if_i\pi_M^\ast(\omega^i)+\rho\wedge\pi_\I^\ast(dt)\right)\\
&=&\iota_b^\ast\left(\sum_if_i\pi_M^\ast(\omega^i)+\rho\wedge\pi_\I^\ast(dt)\right)-\iota_a^\ast\left(\sum_if_i\pi_M^\ast(\omega^i)+\rho\wedge\pi_\I^\ast(dt)\right)\\
&=&\sum_i\left(\iota_b^\ast(f_i)\omega^i-\iota_a^\ast(f_i)\omega^i\right).
\end{eqnarray*}
\end{proof}

\begin{corollary}
%
Let $\omega\in\Lambda(M\times [a,b])$. Then
\begin{equation}\label{BellissimaFormula}
\int_a^b \omega'_t = \omega_b-\omega_a.
 \end{equation}
 \end{corollary}
\begin{proof}
First observe that  $\omega_b-\omega_a=(\iota_b^*-\iota_a^*)(\omega)$. On the other hand, the family $\{\omega'_t\}$ corresponds to the differential form $L_{\frac{\dd }{\dd t}}(\omega)$ (see Corollary \ref{corNablaXLX}) and hence
\begin{equation}\label{eqIndicataDalProfConUnSimboloDiREsistenzaElettrica}
\int_a^b \omega'_t = {\III_a^b}(\overline{L_{\frac{\dd }{\dd t}}(\omega)}) = ({\III_a^b}\circ p_0)(\overline{L_{\frac{\dd }{\dd t}}(\omega)})=({\III_a^b}\circ p_0\circ L_{\frac{\dd }{\dd t}})(\omega),
\end{equation}
%
Moreover, 
\begin{equation}\label{eqCommNablaLieEOrizzontalizzazione}
p_0 \circ L_{\frac{\partial}{\partial t}}=\nabla_{\frac{\partial}{\partial t}}\circ p_0.
\end{equation}
Indeed, left and right hand sides operators restricted to horizontal forms coincide with  $\nabla_{\frac{\partial}{\partial t}}$.  Also, these operators annihilate vertical differential forms, since $L_{\frac{\partial}{\partial t}}$ preserves the class of vertical forms (see Proposition \ref{propVertLX}) while $p_0$ annihilates it.\par
Now it follows from \REF{eqIndicataDalProfConUnSimboloDiREsistenzaElettrica} and Lemma \ref{lemLemmaCheNonEraUnLemmaMaLoEDiventato} that
\begin{equation}\label{}
\int_a^b \omega'_t =  ({\III_a^b}\circ p_0\circ L_{\frac{\dd }{\dd t}})(\omega)=  ({\III_a^b}\circ  \nabla_{\frac{\partial}{\partial t}}\circ p_0)(\omega)=(\iota_b^*-\iota_a^*)(\omega)=\omega_b-\omega_a
\end{equation}

\end{proof}
 Formula \REF{BellissimaFormula} is a   generalization of the historical Newton--Leibniz formula, to    smooth homotopies and differential forms. We call it \virg{universal}, since the Newton--Leibniz formula for a particular homotopy $F$ is derived from it by means of $F^*$.\par


%
%
Let now $F:M\times\G\longrightarrow N$ be a smooth homotopy, and $\eta=F^\ast(\omega)$.
\begin{corollary}
It holds
\begin{equation}\label{eqN-LFormuForm}
\int_a^b \eta'_t=\eta_b -\eta_a.
\end{equation}
\end{corollary}
\begin{proof}
A particular case of \REF{BellissimaFormula}, where $\omega$ is replaced by $F^\ast(\omega)$.
\end{proof}


\begin{remark}
Notice that \REF{eqN-LFormuForm} may be read as  $({\III_a^b}\circ \nabla_{\frac{\partial}{\partial t}}\circ p_0 \circ F^\ast )(\omega)=(F_b^\ast-F_a^\ast)(\omega)$. Since $\omega$ is arbitrary, this implies 
 \begin{equation}\label{eqN-LFormuFormALG-particolare}
{\III_a^b}\circ \nabla_{\frac{\partial}{\partial t}}\circ p_0 \circ F^\ast =F_b^\ast-F_a^\ast.
\end{equation}
  In its turn, \REF{eqN-LFormuFormALG-particolare} is obtained from   \REF{eqN-LFormuFormALG}, by composing  on the right the latter with $F^\ast$. This shows the universality  of  \REF{eqN-LFormuFormALG}.
\end{remark}

%
%
%
%
Let $\{\eta^F_t\}$ be the family of forms on $M$ determined by  $i_{\frac{\dd F}{\dd t}}(\omega)=i_{\frac{\dd }{\dd t}}(F^*(\omega))$. The operator $$h^F:\Lambda(N)\longrightarrow\Lambda(M),\quad \omega \longmapsto \int_a^b \eta^F_t$$ is called the \emph{homotopy operator} associated with the homotopy $F$. Equivalently,
\begin{equation}\label{eqIndicataDalProfConUnSimboloDiREsistenzaElettricaBarrata}
h^F={\III_a^b}\circ  p_0 \circ i_{\frac{\dd F}{\dd t}}.
\end{equation}
Obviously, $h^F$ is a linear operator of degree $-1$.
\begin{theorem}\label{ThHomotopyFormula}
The following homotpy formula takes place.
\begin{equation}\label{eqFurmulaOmotopia}
F_b^\ast-F_a^\ast=[h^F,d].
\end{equation}
\end{theorem}
\begin{proof}
By combining \REF{eqN-LFormuFormALG-particolare} and \REF{eqCommNablaLieEOrizzontalizzazione} we have
\begin{align}
F_b^\ast-F_a^\ast ={\III_a^b}\circ \nabla_{\frac{\dd}{\dd t}}\circ p_0 \circ F^\ast&={\III_a^b}\circ  p_0 \circ L_{\frac{\dd}{\dd t}}\circ F^*\label{eqEquazioneZeroUndiciDElleVEcchiePagine}\\
&={\III_a^b}\circ p_0 \circ \left( i_{\frac{\dd}{\dd t}}\circ d \circ F^\ast + d\circ  i_{\frac{\dd}{\dd t}}\circ F^\ast\right)  \nonumber\\
&={\III_a^b}\circ p_0 \circ \left( i_{\frac{\dd}{\dd t}}\circ F^\ast\circ d_{N} + d_{M\times [a,b]}\circ  i_{\frac{\dd}{\dd t}}\circ F^\ast\right)  \nonumber\\
&={\III_a^b}\circ p_0 \circ  i_{\frac{\dd F}{\dd t}} \circ d_{N} +{\III_a^b}\circ p_0\circ d_{M\times [a,b]}\circ  i_{\frac{\dd F}{\dd t}}  .\nonumber
\end{align}
On the other hand Corollary \ref{corHorizOp} shows that the  composition
\begin{equation*}
\Lambda^\ast(M\times [a,b])\stackrel{p_0}{\longrightarrow}\Lambda_{\pi_M}^\ast(M\times [a,b])\stackrel{{\III_a^b}}{\longrightarrow}\Lambda^\ast(M)
\end{equation*}
is a  cochain map, i.e., $\III_a^b\circ p_0\circ d_{M\times [a,b]}=d_M\circ\III_a^b\circ p_0$. This fact allows to rewrite formula \REF{eqEquazioneZeroUndiciDElleVEcchiePagine} as
\begin{equation*}
F^\ast_b-F^\ast_a=h^F\circ d_N + d_M \circ h^F.
\end{equation*}
\end{proof}

\section*{Acknowledgements}
This research was partially carried out during the author  post--doc fellowship at the    Department of Mathematics of the  University of Salerno, and was partially supported by       INFN (Gruppo collegato di Salerno). The author    is also thankful to the {\it Istituto Italiano per gli Studi Filosofici}   for indispensable financial support.

\end{document}